\documentclass[10pt]{amsart}
\usepackage[margin=1.7in]{geometry}
\usepackage{color}
\newtheorem{theorem}{Theorem}[section]
\newtheorem{lemma}[theorem]{Lemma}

\theoremstyle{definition}

\usepackage{mathtools}
\theoremstyle{remark}

\numberwithin{equation}{section}

\usepackage{graphicx}
\DeclareGraphicsExtensions{.pdf,.png,.jpg}
\begin{document} 
	\title[Relativistic BGK model in the FLRW spacetime]{Relativistic BGK model for massless particles in the FLRW spacetime}
	\author[B.-H. Hwang]{Byung-Hoon Hwang}
	\address{Department of Mathematics, Sungkyunkwan University, Suwon 440-746, Republic of Korea}
	\email{bhh0116@skku.edu}
	
	\author[H. Lee]{Ho Lee}
	\address{Department of Mathematics and Research Institute for Basic Science, Kyung Hee University, Seoul, 02447, Republic of Korea}
	\email{holee@khu.ac.kr}
	
	\author[S.-B. Yun]{Seok-Bae Yun}
	\address{Department of Mathematics, Sungkyunkwan University, Suwon 440-746, Republic of Korea}
	
	\email{sbyun01@skku.edu}
	
	\begin{abstract}	
			In this paper, we address the Cauchy problem for the relativistic BGK model proposed by Anderson and Witting for massless particles in the Friedmann-Lema{\^i}tre-Robertson-Walker (FLRW) spacetime.   
	\end{abstract}
	\maketitle \vspace{-0.8cm}
	\section{Introduction} 
	\subsection{Relativistic BGK model}	
We consider the Anderson and Witting type relativistic BGK model \cite{AW}:
\begin{equation}\label{Original AW}
p^\alpha\frac{\partial F}{\partial x^\alpha} -\Gamma^\alpha_{\beta\gamma}p^\beta p^\gamma \frac{\partial F}{\partial p^\alpha}=-\frac{U^\mu p_\mu}{c^2\tau}\left(J(F)-F \right).
\end{equation}   
Here $F\equiv F(x^\alpha,p^\alpha)$ is the distribution function representing a number density of particles in the phase space spanned by the spacetime coordinates $x^\alpha$ and four-momentum $p^\alpha$.
 The Greek indices run from 0 to 3 and the repeated Greek indices are assumed to be summed over its whole range. In \eqref{Original AW}, $\Gamma^\alpha_{\beta\gamma}$ denote the Christoffel symbols, $U^\mu$ is the Landau-Lifshitz four-velocity, $c$ is the speed of light, $\tau$ is the characteristic time of order of the mean free time, and $J(F)$ is the J\"{u}ttner distribution \cite{Juttner,Juttner 2} which takes the form of 
$$
J(F)=\begin{cases}
\frac{g_s}{h^3}\exp\left\{\frac{\mu_E}{kT_E}+\frac{U_E^{\alpha} p_\alpha}{kT_E}\right\} & \text{for the relativistic Maxwell-Boltzmann statistics},\cr
\frac{g_s/h^3}{\exp\left\{-\frac{\mu_E}{kT_E}-\frac{U_E^{\alpha} p_\alpha}{kT_E}\right\}-1} & \text{for the relativistic Bose-Einstein statistics},
\end{cases}
$$
where $h$ is the Planck constant, $k$ is the Boltzmann constant, $g_s$ is the degeneracy factor, $\mu_E$ is the chemical potential, $U_E^\alpha=(\sqrt{c^2+|U_E|^2},U_E)$ is the four-velocity, and $T_E$ is the equilibrium temperature. 
Note that $\mu_E$, $U_E$ and $T_E$ are functions of $t$ and $x$ determined by the relations
\begin{align*} 
		U^\mu\int_{\mathbb{R}^3}p_\mu\left(J(F)-F \right)\sqrt{-|\eta|}\, \frac{dp}{p^0}=0,\qquad  U^\mu\int_{\mathbb{R}^3} p_\mu p^\nu\left(J(F)-F \right)\sqrt{-|\eta|}\, \frac{dp}{p^0}=0,
 \end{align*}
where $|\eta|$ denotes the determinant of the metric tensor $\eta_{\alpha\beta}$ and $dp = dp^1 dp^2 dp^3$, so that the conservation laws of particle number, momentum and energy for \eqref{Original AW} hold true
\begin{align}\label{conservation laws}
	{N^\alpha}_{;\alpha}=\partial_\alpha N^\alpha+\Gamma^{\alpha}_{\alpha\mu}N^\mu=0,\qquad {T^{\alpha\beta}}_{;\alpha}=\partial_\alpha T^{\alpha\beta}+\Gamma^{\alpha}_{\alpha\mu}T^{\mu\beta}+\Gamma^{\beta}_{\alpha\mu}T^{\alpha\mu}=0.
\end{align}

The BGK model \cite{BGK,Wal} is the most well-known model equation of the classical Boltzmann equation. Three relativistic generalizations have been proposed respectively by Marle \cite{Mar1,Mar3}, Anderson and Witting \cite{AW}, and recently by Pennisi and Ruggeri \cite{PR}, which have been widely applied to various physical problems \cite{FRS,FRS2,JRS,MKSH,MNR,MW}. The first mathematical analysis for relativistic BGK models was carried out in \cite{BCNS} where the unique determination of equilibrium variables, the scaling limits, and the linearization problem were studied for the Marle model. For the existence theory of the nonlinear Marle model, we refer to \cite{BNU} for near-equilibrium solutions, \cite{HY1} for stationary solutions, and \cite{CJS} for weak solutions. In the case of the Anderson-Witting model, the unique determination of equilibrium variables, and the global existence and asymptotic behavior of near-equilibrium solutions were studied in \cite{HY2}, and the stationary problem was covered in \cite{HY3}. Recently, the Pennisi-Ruggeri model for polyatomic gases was studied in \cite{HRY} where the unique determination of equilibrium variables, and the global existence and asymptotic behavior of near-equilibrium solutions were addressed.  To the best knowledge of authors, the BGK model has not been much studied in the context of cosmology, which is the main motivation of the current work.


\subsection{Reduction in the case of the  FLRW  spacetime}
In this paper, we are concerned with the Cauchy problem of the Anderson-Witting model \eqref{Original AW}   for massless particles in an isotropic and spatially homogeneous spacetime, namely the Friedmann-Lema{\^i}tre-Robertson-Walker (FLRW) spacetime. As such, the distribution function is also assumed to be isotropic and spatially homogeneous. Throughout the paper, we set all physical constants to be unity except for particle mass $(m=0)$ for brevity. In the FLRW case, the metric tensor $\eta_{\mu\nu}$ is given by
$$
\eta_{00}=\eta^{00}=-1,\qquad \eta_{ij} = R^2\delta_{ij},\qquad \eta^{ij} = R^{-2}\delta^{ij}
$$
for $1\leq i,j\leq 3$, where $\delta_{ij}$ and $\delta^{ij}$ are the Kronecker delta, and $\eta^{\mu\nu}$ is the inverse of $\eta_{\mu\nu}$. Hence the four momentum $p^\alpha$ and its covariant components $v_\alpha:=\eta_{\alpha\beta} p^\beta(=p_\alpha)$ are defined by 
\begin{equation}\label{contra co}
p^\alpha=\left(R|p|,p\right), \qquad v_\alpha=\left(-p^0,R^2p\right)=\left(-R^{-1}|v|, v \right),
\end{equation}
due to the mass shell condition  $$\eta_{\alpha\beta}p^\alpha p^\beta =-m^2 = \eta^{\alpha\beta} v_\alpha v_\beta$$ with $m=0$. Here  $R=R(t)$ denotes the cosmic scale factor that is determined by the following two equations (called the Friedmann and the acceleration equations respectively)
$$
\left(\frac{\dot{R}}{R}\right)^2=\frac{8\pi }{3}en,\qquad \frac{\ddot{R}}{R}=-\frac{4\pi }{3}\left\{en+3P \right\}
$$
where the dot denotes the derivative with respect to $t$. In the massless case, the scale factor is easily solved from the above equations since 
$$
\left(\frac{\dot{R}}{R}\right)^2 +\frac{\ddot{R}}{R}=0
$$ (see \eqref{contra macroscopic} below) so that it depends only on $R(0)$ and $\dot{R}(0)$. Hence, we may assume that the scale factor is given by 
$$
R = C(t +t_0)^{1/2}
$$ for some positive constants $C$ and $t_0$.
To define the macroscopic quantities $e,n$ and $P$, we consider the particle four-flow $N^\alpha$ and energy-momentum tensor $T^{\alpha\beta}$: 
\begin{equation*}
N^\alpha=\int_{\mathbb{R}^3} p^\alpha F\,\sqrt{-|\eta|}\frac{dp}{p^0},\qquad  T^{\alpha\beta}=\int_{\mathbb{R}^3} p^\alpha p^\beta F\,\sqrt{-|\eta|}\frac{dp}{p^0}.
\end{equation*}
In  the Landau-Lifshitz frame \cite{LL} with the isotropy assumption, both $N^\alpha$ and $T^{\alpha\beta}$ are decomposed as
$$
N^\alpha=nU^\alpha,\qquad T^{\alpha\beta}=(en+p)U^\alpha U^\beta+p\eta^{\alpha\beta}.
$$
Here the particle number density $n$, the Landau-Lifshitz four-velocity $U^\alpha$, the internal energy per particle $e$, and the pressure $P$ are defined as follows
\begin{align}\label{contra macroscopic}
\begin{split}
n&=R^3\int_{\mathbb{R}^3}F \,dp,\cr 
U ^\alpha &=\frac{R^3}{n}\int_{\mathbb{R}^3} p^\alpha F\, \frac{dp}{p^0}=(1,0,0,0),\cr 
e  &=\frac{R^4}{n}\int_{\mathbb{R}^3}|p| F \,dp =R\frac{\int_{\mathbb{R}^3}|p| F \,dp}{\int_{\mathbb{R}^3}F \,dp},\cr 
P &=\frac{R^5}{3}\int_{\mathbb{R}^3}|p|^2F\,\frac{dp}{p^0},
\end{split}\end{align}
where we used $\sqrt{-|\eta|} = R^3$ and \eqref{contra co}. In the FLRW case, the nonzero Christoffel symbols are
\begin{equation*}
\Gamma^0_{ij}=R\dot{R}\delta_{ij},\qquad\Gamma^{i}_{j0}=\Gamma^i_{0j}=\frac{\dot{R}}{R}\delta^i_j
\end{equation*}
for $1\leq i,j\leq 3$. In conclusion, the Cauchy problem of the Anderson-Witting model \eqref{Original AW} for massless particles in the FLRW spacetime is reduced into
\begin{align}\label{AW in RW}\begin{split}  
 		\partial_tF-2\frac{\dot{R}}{R}p\cdot\nabla_p F&= J(F)-F,\cr
 			F(0,p)&=F_0(p)
\end{split} \end{align} 
where the scale factor is
$$
R = C(t +t_0)^{1/2}.
$$ 
We then consider the characteristic curve of \eqref{AW in RW}:
\begin{equation}\label{characteristic}
\frac{dp}{dt}=-2\frac{\dot{R}}{R}p,\qquad p(0)=y
\end{equation}
which can be solved explicitly: 
$$
p(t)=R^{-2}(t)y.
$$
Therefore, in terms of the covariant variable $v=R^2(t)p$,  \eqref{AW in RW} can be simplified further into
\begin{align}\label{AW in RW 2}\begin{split}  
		\partial_t F(t,v)&= J(F)(t,v)-F(t,v),\cr	
		F(0,v)&=F_0(v)
\end{split} \end{align}
where the J\"{u}ttner distribution $J(F)$ is written as
\begin{equation}\label{general J(F) 2} 
J(F)= 
\exp\left\{\frac{\mu_E}{T_E}+\frac{U_E^{\alpha} v_\alpha}{T_E}\right\}
\end{equation}
for the relativistic Maxwell-Boltzmann statistics, and
\begin{equation}\label{general J(F) 3} 
J(F)=
\frac{1}{\exp\left\{-\frac{\mu_E}{T_E}-\frac{U_E^{\alpha} v_\alpha}{T_E}\right\}-1} 
\end{equation}
for the relativistic Bose-Einstein statistics.
The Cauchy problem \eqref{AW in RW 2} is not as simple as it looks: Due to the specific structure of $J(F)$, the Cauchy problem \eqref{AW in RW 2} must be understood to be coupled to the nonlinear relation \eqref{v cgamma} in the case of the relativistic Bose-Einstein statistics:
\begin{align*}   
	&\hspace{0.5cm}\partial_tF= \frac{1}{\exp\{c+\gamma|v|\}-1}-F,\cr	
		&\hspace{1.5cm}\beta(c )=\frac{\rho(F) }{\left(3T(F) \right)^3},\\ 
		&\gamma =\left\{\rho(F)\right\} ^{-\frac{1}{3}}\left(\int_{\mathbb{R}^3} \frac{1}{\exp\left\{c +|v|\right\}-1}   \,dv\right)^{\frac{1}{3}}, 
 \end{align*}
  which makes the problem \eqref{AW in RW 2} a highly nontrivial one (See Lemma \ref{form of J(F) 2} and (\ref{v AW 2})).\noindent\newline

To state our main result, we define two types of global equilibriums
\begin{equation}\label{both J^0}
J^0=\begin{cases}
\exp\left\{-|v|\right\} & \text{for the relativistic Maxwell-Boltzmann statistics},\cr
\frac{1}{\exp\left\{1+|v|\right\}-1} & \text{for the relativistic Bose-Einstein statistics. }
\end{cases}
\end{equation}
Our main result is as follows.
\begin{theorem}\label{main}
Assume $F_0 = F_{0}(v)$ is nonnegative, and $F_0$ and $J^0$ share the total particle number and energy in the following sense: 
	\begin{equation}\label{initial} 
		\int_{\mathbb{R}^3} F_0\,dv=\int_{\mathbb{R}^3} J^0\,dv,\qquad \int_{\mathbb{R}^3} |v|F_0\,dv=\int_{\mathbb{R}^3} |v|J^0\,dv.
	\end{equation} 	
Then the Cauchy problem \eqref{AW in RW 2} is explicitly solved as follows
\begin{align*}
	F(t,v)&=\exp(-t)F_0(v)+\left\{1-\exp(-t)\right\}J^0.
	\end{align*}
\end{theorem}
As is mentioned above, mathematical analysis for BGK models in cosmological framework has never been made in the literature. We refer to \cite{MW} for the study of exact solutions and to \cite{baz1,baz2} in the Boltzmann case. For an introduction to the Boltzmann or Vlasov equation for massless particles we refer to \cite{BFH,LNT,LNT2,T03}.\newline

This paper is organized as follows. In Section 2, we find an explicit form of J\"{u}ttner distribution of \eqref{general J(F) 2} and \eqref{general J(F) 3} for massless particles in the FLRW spacetime. In Section 3, we investigate the iteration scheme for \eqref{AW in RW 2} to prove Theorem \ref{main}. 
\noindent
\newline

\section{Determination of $J(F)$}
Recall from \eqref{general J(F) 2} and \eqref{general J(F) 3} that  the J\"{u}ttner distribution $J(F)$  has unknown variables $\mu_E, U_E$ and $T_E$, and using the covariant variable $v$ of \eqref{contra co}, unknown variables are determined through the relation
\begin{align}\label{cancellation 2}\begin{split}
U^\alpha\int_{\mathbb{R}^3}v_\alpha\left(J(F)-F \right)\frac{1}{\sqrt{-|\eta|}}\, \frac{dv}{v_0}&=0,\cr 
U^\alpha\int_{\mathbb{R}^3} v_\alpha \eta^{\mu\beta}v_\mu\left(J(F)-F \right)\frac{1}{\sqrt{-|\eta|}}\, \frac{dv}{v_0}&=0
\end{split}\end{align}
so that the conservation laws \eqref{conservation laws} hold true. In the following lemma, we investigate the explicit form of $J(F)$ using the relation \eqref{cancellation 2} for the case of massless particles in the FLRW spacetime.

\begin{lemma}\label{form of J(F)}
The explicit form of \eqref{general J(F) 2} satisfying \eqref{cancellation 2} is given by
\begin{equation*}
	J(F)=
		\frac{\rho  }{8\pi T ^3} \exp\left\{-\frac{|v|}{T }\right\} 
	\end{equation*}
where $\rho $ and $T $ denote 
\begin{align*}
	\rho = \int_{\mathbb{R}^3}F \,dv,\qquad 3 T =   \frac{\int_{\mathbb{R}^3}|v| F \,dv}{\int_{\mathbb{R}^3}F \,dv}.
\end{align*}
\end{lemma} 

\begin{proof} Taking $U_E^\alpha=U^\alpha$ (see \eqref{contra macroscopic}), then \eqref{general J(F) 2} becomes
	\begin{equation}\label{v J(F)}
	J(F)= 
	\exp\left\{\frac{\mu_E}{T_E}-\frac{|v|}{RT_E}\right\}
	\end{equation}
	and \eqref{cancellation 2} reduces to
	\begin{align}\label{cancel 2}\begin{split}
	 \int_{\mathbb{R}^3}J(F)\, dv&= \int_{\mathbb{R}^3}F\, dv,\cr
	    \left(\int_{\mathbb{R}^3}  |v| J(F) \, dv,0,0,0\right)&=  \left(\int_{\mathbb{R}^3}  |v| F \, dv,0,0,0\right)
\end{split}	\end{align}
	due to the isotropic property of $F$. Inserting \eqref{v J(F)} into \eqref{cancel 2}, we have
	\begin{align*}
		8\pi R^3T_E^3 \exp\left\{\frac{\mu_E}{T_E}\right\}&=\rho,\cr 
		24\pi R^4T_E^4\exp\left\{\frac{\mu_E}{T_E}\right\}&= 3 \rho T,
	\end{align*}
which leads to
	\begin{equation}\label{J 1}
T_E=R^{-1}T,\qquad \exp\left\{\frac{\mu_E}{T_E} \right\}=\frac{\rho}{8\pi T^3}.
\end{equation}
	Putting \eqref{J 1} into \eqref{v J(F)} gives the desired result.
\end{proof}

\begin{lemma}\label{form of J(F) 2}
The explicit form of \eqref{general J(F) 3}  satisfying \eqref{cancellation 2} is given by
	\begin{equation*}
		J(F)=			\frac{1}{\exp\left\{c +\gamma  |v|\right\}-1}  
 	\end{equation*}
	where $c $ and $\gamma $ are determined by the relations
	\begin{align}\label{v cgamma}\begin{split}
		&\beta(c )\equiv\frac{\left\{\int_{\mathbb{R}^3} \frac{1}{\exp\left\{c +|v|\right\}-1}   \,dv \right\}^4}{\left\{\int_{\mathbb{R}^3} \frac{|v|}{\exp\left\{c +|v|\right\}-1}   \,dv \right\}^3}=\frac{\rho }{(3T )^3},\cr 
		&\gamma =\rho ^{-\frac{1}{3}}\left(\int_{\mathbb{R}^3} \frac{1}{\exp\left\{c +|v|\right\}-1}   \,dv\right)^{\frac{1}{3}}.
\end{split}	\end{align}
	If $c$ is positive, and $F$ satisfies  
	\begin{equation}\label{Apery 2}
	0<\frac{\rho }{(3T )^3}<    \frac{8\pi}{27}\frac{\left\{\sum_{k=1}^\infty \frac{1}{k^{3}}\right\}^4}{\left\{\sum_{k=1}^\infty \frac{1}{k^{4}}\right\}^3},
	\end{equation}
	then \eqref{v cgamma} determines unique $c$ and $\gamma$.
\end{lemma} 

\begin{proof} $\bullet$ The explicit form of \eqref{general J(F) 3}: As in \eqref{v J(F)} of Lemma \ref{form of J(F)}, we can see that $J(F)$  satisfying \eqref{cancellation 2} is given by
	\begin{equation}\label{v J(F) 2}
	J(F)= \frac{1}{\exp\left\{c+\gamma |v|\right\}-1},  	
	\end{equation}
	where $c$ and $\gamma$ denote
	$$
	c=-\frac{\mu_E}{T_E}, \qquad \gamma=\frac{1}{RT_E}.
	$$
Inserting \eqref{v J(F) 2} into \eqref{cancel 2} gives 
	\begin{align}\label{cgamma 2}\begin{split}
		&\gamma^3\int_{\mathbb{R}^3}\frac{1}{\exp\left\{c+|v| \right\}-1}\,dv=\rho,\cr
		& 	\gamma^4\int_{\mathbb{R}^3}\frac{|v|}{\exp\left\{c+|v| \right\}-1}\, dv=3\rho T
	\end{split}\end{align}
	from which we can derive \eqref{v cgamma}.
	\noindent\newline
	$\bullet$ Unique determination of $c$ and $\gamma$: The integrals of the Bose-Einstein distribution can be represented by
	\begin{equation}\label{BE integral}
	\int_0^\infty \frac{r^n}{\exp\left\{r  \right\}/z-1}\,dr=n! \sum_{k=1}^\infty \frac{z^k}{k^{n+1}}
	\end{equation}
	which converges for $n\in\mathbb{N}$ and $z<1$. For $c>0$, we put $z=e^{-c}$ and apply \eqref{BE integral} to $\beta(c)$ of \eqref{v cgamma} to see 
	\begin{align}\begin{split}
	\beta(c)&= 4\pi \frac{\left\{\int_{0}^\infty\frac{r^2}{\exp\left\{c+r \right\}-1}\,dr\right\}^4}{\left\{\int_{0}^\infty\frac{r^3}{\exp\left\{c+r \right\}-1}\,dr\right\}^3}= 4\pi\frac{\left\{2!\sum_{k=1}^\infty \frac{\exp\{-ck\}}{k^{3}}\right\}^4}{\left\{3!\sum_{k=1}^\infty \frac{\exp\{-ck\}}{k^{4}}\right\}^3}\cr
	&=\frac{8\pi}{27} \frac{\left\{\sum_{k=1}^\infty \frac{\exp\{-ck\}}{k^{3}}\right\}^4}{\left\{\sum_{k=1}^\infty \frac{\exp\{-ck\}}{k^{4}}\right\}^3}.
	\end{split}\end{align}
	Now we	differentiate $\beta(c)$ with respect to $c$ and use the following relation 
	\begin{equation*}
		\frac{\partial}{\partial c}\left\{\sum_{k=1}^\infty \frac{\exp\{-ck\}}{k^{n}}\right\}=- \sum_{k=1}^\infty \frac{\exp\{-ck\}}{k^{n-1}}
	\end{equation*}
	to obtain
	\begin{align*}
		&\frac{d}{dc}	\beta(c)\cr
		&=\frac{8\pi}{27} \left[-4\sum_{k=1}^\infty \frac{\exp\{-ck\}}{k^{2}}\frac{\left\{\sum_{k=1}^\infty \frac{\exp\{-ck\}}{k^{3}}\right\}^3}{\left\{\sum_{k=1}^\infty \frac{\exp\{-ck\}}{k^{4}}\right\}^3}+3\sum_{k=1}^\infty \frac{\exp\{-ck\}}{k^{3}}\frac{\left\{\sum_{k=1}^\infty \frac{\exp\{-ck\}}{k^{3}}\right\}^4}{\left\{\sum_{k=1}^\infty \frac{\exp\{-ck\}}{k^{4}}\right\}^4}  \right]  	\cr
		&=  \beta(c)\left[-4\frac{\sum_{k=1}^\infty \frac{\exp\{-ck\}}{k^{2}}}{\sum_{k=1}^\infty \frac{\exp\{-ck\}}{k^{3}}}+3\frac{\sum_{k=1}^\infty \frac{\exp\{-ck\}}{k^{3}}}{\sum_{k=1}^\infty \frac{\exp\{-ck\}}{k^{4}}}  \right]\cr
		&=   \beta(c)\left\{\frac{3\left(\sum_{k=1}^\infty \frac{\exp\{-ck\}}{k^{3}}\right)^2-4\sum_{k=1}^\infty \frac{\exp\{-ck\}}{k^{2}}\sum_{k=1}^\infty \frac{\exp\{-ck\}}{k^{4}}}{\sum_{k=1}^\infty \frac{\exp\{-ck\}}{k^{3}}\sum_{k=1}^\infty \frac{\exp\{-ck\}}{k^{4}}} \right\}
	\end{align*}
	which is strictly negative since
	\begin{align*}
		\left(\sum_{k=1}^\infty \frac{\exp\{-ck\}}{k^{3}}\right)^2\le \sum_{k=1}^\infty \frac{\exp\{-ck\}}{k^{2}}\sum_{k=1}^\infty \frac{\exp\{-ck\}}{k^{4}}.
	\end{align*}
	Thus $\beta(c)$ is stricitly decreasing on $c\in (0,\infty)$. On the other hand, the limiting behavior of $\beta(c)$ is 
	\begin{align*}
	\lim_{c\rightarrow 0}\beta(c)&=\frac{8\pi}{27}\lim_{c\rightarrow 0} \frac{\left\{\sum_{k=1}^\infty \frac{\exp\{-ck\}}{k^{3}}\right\}^4}{\left\{\sum_{k=1}^\infty \frac{\exp\{-ck\}}{k^{4}}\right\}^3}=\frac{8\pi}{27}\frac{\left\{\sum_{k=1}^\infty \frac{1}{k^{3}}\right\}^4}{\left\{\sum_{k=1}^\infty \frac{1}{k^{4}}\right\}^3},
\end{align*}
and
	\begin{align*}
\lim_{c\rightarrow\infty}\beta(c)&=\frac{8\pi}{27}\lim_{c\rightarrow\infty} \frac{\left\{\sum_{k=1}^\infty \frac{\exp\{-ck\}}{k^{3}}\right\}^4}{\left\{\sum_{k=1}^\infty \frac{\exp\{-ck\}}{k^{4}}\right\}^3}=0.
\end{align*}
So we can see that the range of $\beta(c)$ is $\left(0,\lim_{c\rightarrow 0}\beta(c)\right)$ on $c\in (0,\infty)$. Since we have assumed \eqref{Apery 2}, one finds
	$$
	0<\frac{\rho}{(3T)^3}< \lim_{c\rightarrow 0}\beta(c).
	$$
	Therefore, there is a one-to-one correspondence between $c$ and $\rho/(3T)^3$ of \eqref{v cgamma}, which guarantees the unique determination of $c$ and $\gamma$.
\end{proof}

\section{Proof of Theorem \ref{main} }
The proof of Theorem \ref{main} for the relativistic Maxwell-Boltzmann statistics is straightforward. Recall from Lemma \ref{form of J(F)} that $J(F)$ for the relativistic Maxwell-Boltzmann statistics is defined by
$$
J(F)=	\frac{\rho }{8\pi T^3} \exp\left\{-\frac{|v|}{T}\right\}
$$
with
$$
\rho =\int_{\mathbb{R}^3}F \,dv,\qquad 3 T=\frac{\int_{\mathbb{R}^3}|v| F \,dv}{\int_{\mathbb{R}^3}F \,dv}
$$
and it satisfies \eqref{cancel 2}:
\begin{equation}\label{conservation}
\int_{\mathbb{R}^3} J(F)-F\,\,dv=0,\qquad \int_{\mathbb{R}^3} |v|(J(F)-F)\,\,dv=0.
\end{equation}
In terms of \eqref{AW in RW 2}, this leads to the following conservation laws
\begin{equation}\label{cancell}
\frac{d}{dt}\int_{\mathbb{R}^3} F  \,dv=0,\qquad \frac{d}{dt}\int_{\mathbb{R}^3}|v| F  \,dv=0.
\end{equation}
Since we assumed
\begin{align*}
\int_{\mathbb{R}^3} F_0\,dv&=\int_{\mathbb{R}^3} J^0\,dv(=8\pi),\cr
 \int_{\mathbb{R}^3} |v|F_0\,dv&=\int_{\mathbb{R}^3} |v|J^0\,dv(=24\pi),
\end{align*}
combining \eqref{conservation} and \eqref{cancell} gives
$$
\rho=8\pi ,\qquad T=1,  
$$
and hence
$$
J(F)=J^0.
$$
Therefore \eqref{AW in RW 2} is explicitly solved as 
\begin{align*}
	F(t,v)&=\exp(-t)F_0(v)+\big(1-\exp(-t)\big)J^0.
\end{align*}

In the case of relativistic Bose-Einstein statistics, it is not obvious because, by Lemma \ref{form of J(F) 2} $J(F)$ is well-defined in a way to satisfiy \eqref{conservation} only when the solution $F$ satisfies \eqref{Apery 2}.
For this, we consider the following iteration scheme that for $n\ge 0$,
\begin{align}\label{v AW 2}\begin{split}  
		 &\hspace{0.5cm}\partial_tF^{n+1}= \frac{1}{\exp\{c^n+\gamma^n|v|\}-1}-F^{n+1},\cr	
		 &\hspace{1.5cm}\beta(c^n )=\frac{\rho(F^n) }{\left(3T(F^n) \right)^3},\cr
		 & \gamma^n =\left\{\rho(F^n)\right\} ^{-\frac{1}{3}}\left(\int_{\mathbb{R}^3} \frac{1}{\exp\left\{c^n +|v|\right\}-1}   \,dv\right)^{\frac{1}{3}}, 
\end{split} \end{align}
where we set $F^{n}(0,v)=F_0(v)$ and $c^n >0$.
Then it follows from \eqref{initial} that
$$
		\int_{\mathbb{R}^3} F^0\,dv=\int_{\mathbb{R}^3} J^0\,dv,\qquad \int_{\mathbb{R}^3} |v|F^0\,dv=\int_{\mathbb{R}^3} |v|J^0\,dv.
$$
Inserting $F^0$ into the definition of $\rho/(3T)^3$, this gives
\begin{equation*} 
0< \frac{\rho(F^0)}{\left\{3T(F^0)\right\}^3} =\frac{\left(\int_{\mathbb{R}^3}F^0\,dv\right)^4}{\left(\int_{\mathbb{R}^3}|v|F^0\,dv\right)^3}=\frac{\left\{\int_{\mathbb{R}^3} J^0\,dv \right\}^4}{\left\{\int_{\mathbb{R}^3} |v|J^0\,dv \right\}^3}=\beta(1).
\end{equation*}
Since $\beta(c)$ is strictly decreasing on $c\in(0,\infty)$ (see the proof of Lemma \ref{form of J(F) 2}), we have
$$
0< \frac{\rho(F^0)}{\left\{3T(F^0)\right\}^3}=\beta(1) <\lim_{c\rightarrow 0}\beta(c)=\frac{8\pi}{27}\frac{\left\{\sum_{k=1}^\infty \frac{1}{k^{3}}\right\}^4}{\left\{\sum_{k=1}^\infty \frac{1}{k^{4}}\right\}^3},
$$
which says that $F^0$ satisfies \eqref{Apery 2}. So we conclude by Lemma \ref{form of J(F) 2} that $c^0$ is equal to $1$ and
$$
\gamma^0=\left(\int_{\mathbb{R}^3}J^0\,dv\right)^{-\frac{1}{3}}\left(\int_{\mathbb{R}^3} \frac{1}{\exp\left\{1+|v|\right\}-1}   \,dv\right)^{\frac{1}{3}}=1,
$$
and hence
$$ 
\frac{1}{\exp\{c^0+\gamma^0|v|\}-1}=J^0.
$$
Thus $F^1$ of \eqref{v AW 2} satisfies
\begin{align*}
	\int_{\mathbb{R}^3}\left(\begin{matrix}
		1\cr
		|v|
	\end{matrix}\right)	F^1(t,v)\,dv &=\exp(-t)\int_{\mathbb{R}^3}\left(\begin{matrix}
		1\cr
		|v|
	\end{matrix}\right)F_0(v)\,dv+\big(1-\exp(-t)\big) \int_{\mathbb{R}^3} \left(\begin{matrix}
		1\cr
		|v|
	\end{matrix}\right)J^0\,dv\cr
	&=\exp(-t)\int_{\mathbb{R}^3}\left(\begin{matrix}
		1\cr
		|v|
	\end{matrix}\right)J^0\,dv+\big(1-\exp(-t)\big) \int_{\mathbb{R}^3}\left(\begin{matrix}
		1\cr
		|v|
	\end{matrix}\right)J^0\,dv\cr
	&= \int_{\mathbb{R}^3}\left(\begin{matrix}
		1\cr
		|v|
	\end{matrix}\right)J^0\,dv.
\end{align*}
Applying the same argument, we get 
$$
\frac{1}{\exp\{c^1+\gamma^1|v|\}-1}=J^0.
$$ 
Proceeding in the same manner leads to
\begin{align*}
	\int_{\mathbb{R}^3}\left(\begin{matrix}
		1\cr
		|v|
	\end{matrix}\right)	F^n(t,v)\,dv= \int_{\mathbb{R}^3}\left(\begin{matrix}
		1\cr
		|v|
	\end{matrix}\right)F_0\,dv,  
\end{align*}
and hence
$$
\frac{1}{\exp\{c^n+\gamma^n|v|\}-1}=J^0
$$
for all $n\in\mathbb{N}$. Therefore, the iteration \eqref{v AW 2} can be explicitly solved as
\begin{align*}
	F^{n+1}(t,v)&=\exp(-t)F_0(v)+\big(1-\exp(-t)\big)J^0
\end{align*}
which gives the desired result. 

%
%

 \noindent\newline
 {\bf Acknowledgements.}
 Byung-Hoon Hwang was supported by Basic Science Research Program through the National Research Foundation of Korea(NRF) funded by the Ministry of Education(No. NRF-2019R1A6A1A10073079). Ho Lee was supported by the Basic Science Research Program through the National Research Foundation of Korea (NRF) funded by the Ministry of Science, ICT \& Future Planning (NRF-2018R1A1A1A05078275). Part of this work was done while Ho Lee was visiting the Laboratoire Jacques-Louis Lions. Seok-Bae Yun was supported by Samsung Science and Technology Foundation under Project Number SSTF-BA1801-02.

\bibliographystyle{amsplain}

\begin{thebibliography}{10}
	\bibitem{AW} Anderson, J. L., Witting, H. R.: A relativistic relaxation-time model for the Boltzmann equation. Physica. {\bf{74}} (1974), 466--488.

\bibitem{BFH}
Barzegar, H. Fajman, D., Hei{\ss}el, G.: Isotropization of slowly expanding spacetimes. Phys. Rev. D {\bf 101} (2020), 044046.

\bibitem{baz1} 
Bazow, D., Denicol, G. S., Heinz, U., Martinez, M., Noronha, J.: Analytic solution of the Boltzmann equation in an expanding system. Phys. Rev. Lett. {\bf 116} (2016), 022301.

\bibitem{baz2} 
Bazow, D., Denicol, G. S., Heinz, U., Martinez, M., Noronha, J.: Nonlinear dynamics from the relativistic Boltzmann equation in the Friedmann-Lema{\^i}tre-Robertson-Walker spacetime. Phys. Rev. D {\bf 94} (2016), 125006.



	\bibitem{BCNS} Bellouquid, A., Calvo, J., Nieto, J., Soler, J.: On the relativistic BGK-Boltzmann model: asymptotics and hydrodynamics. J. Stat. Phys. {\bf 149} (2012), 284--316.
	\bibitem{BNU} Bellouquid, A., Nieto, J., Urrutia, L.: Global existence and asymptotic stability near equilibrium for the relativistic BGK model. Nonlinear Anal. {\bf 114} (2015), 87--104.
	\bibitem{BGK} Bhatnagar, P. L., Gross, E. P. and Krook, M. L.: A model for collision processes in gases. I. Small amplitude processes in charged
	and neutral one-component systems, Phys. Rev. {\bf 94} (1954), 511-525.
	\bibitem{CJS} Calvo, J., Jabin, P.-E., Soler, J.: Global weak solutions to the relativistic BGK equation. Comm. Partial Differential Equations {\bf{45}}(3) (2020), 191--229. 
	
	\bibitem{FRS} Florkowski, W., Ryblewski, R., Strickland, M.: Anisotropic hydrodynamics for rapidly expanding systems. Nucl. Phys. A {\bf{916}} (2013), 249--259.
	\bibitem{FRS2} Florkowski, W., Ryblewski, R., Strickland, M.: Testing viscous and anisotropic hydrodynamics in an exactly solvable case. Phys. Rev. C. {\bf{88}} (2013), 024903.
	
	%
	
	
	
    \bibitem{HRY} Hwang, B.-H., Ruggeri, T., Yun, S.-B.: On a relativistic BGK model for polyatomic gases near equilibrium. Preprint arXiv:2102.00462.
	\bibitem{HY2} Hwang, B.-H., Yun, S.-B.: Anderson-Witting model of the relativistic Boltzmann equation near equilibrium. J. Stat. Phys. {\bf{176}} (2019), 1009--1045.
	\bibitem{HY3} Hwang, B.-H., Yun, S.-B.: Stationary solutions to the Anderson--Witting model of the relativistic Boltzmann equation in a bounded interval. SIAM J. Math. Anal. {\bf 53}(1) (2021), 730--753.
		\bibitem{HY1} Hwang, B.-H., Yun, S.-B.: Stationary solutions to the boundary value problem for the relativistic BGK model in a slab. Kinet. Relat. Models {\bf{12}}(4) (2019), 749--764. 
	\bibitem{JRS} Jaiswal, A., Ryblewski, R., Strickland, M.: Transport coefficients for bulk viscous evolution in the relaxation time approximation. Phys. Rev. C. {\bf{90}} (2014), 044908.
	\bibitem{Juttner} J\"{u}ttner, F.: Das Maxwellsche Gesetz der Geschwindigkeitsverteilung in der Relativtheorie. Ann. Physik and Chemie {\bf{34}} (1911), 856--882.
	\bibitem{Juttner 2} J\"{u}ttner, F.: Die relativistische Quantentheorie des idealen Gases. Zeitschr. Physik {\bf{47}} (1928), 542--566.
	\bibitem{LL} Landau, L. D., Lifshitz, E. M.: Fluid Mechanics. Pergamon Press. (1959).

\bibitem{LNT} Lee, H., Nungesser, E., Tod, P.: The massless Einstein-Boltzmann system with a conformal-gauge singularity in an FLRW background. Classical Quantum Gravity {\bf 37} (2020), no. 3, 035005.

\bibitem{LNT2} Lee, H., Nungesser, E., Tod, K. P.: On the future of solutions to the massless Einstein-Vlasov system in a Bianchi I cosmology. Gen. Relativity Gravitation {\bf 52} (2020), no. 5, 48.

	
	
	
	\bibitem{MW} Maartens, R., Wolvaardt, F. P.: Exact non-equilibrium solutions of the Einstein-Boltzmann equations. Classical Quantum Gravity {\bf{11}} (1994), 203--225.
	\bibitem{Mar3} Marle, C.: Modele cin\'{e}tique pour l\textquoteright \'{e}tablissement des lois de la conduction de la
	chaleur et de la viscosit\'{e} en th\'{e}orie de la relativit\'{e}. C. R. Acad. Sci. Paris {\bf{260}} (1965), 6539--6541.
	\bibitem{Mar1} Marle, C.: Sur l\textquoteright\'{e}tablissement des equations de l\textquoteright hydrodynamique des fluides relativistes dissipatifs, I. L\textquoteright equation de Boltzmann relativiste. Ann. Inst. Henri Poincar\'{e} {\bf{10}} (1969), 67--127.
	
	\bibitem{MKSH} Mendoza, M., Karlin, I., Succi, S., Herrmann, H. J.: Relativistic lattice Boltzmann model with improved dissipation. Phys. Rev. D. {\bf{87}} (2013), 065027.		
	
	\bibitem{MNR} Moln\'{a}r, E., Niemi, H., Rischke, D. H.: Derivation of anisotropic dissipative fluid dynamics from the Boltzmann equation. Phys. Rev. D. {\bf{93}} (2016), 114025.	
	
	\bibitem{PR} Pennisi, S., Ruggeri, T.: A new BGK model for relativistic kinetic theory of monatomic and polyatomic gases. J. Phys. Conf. Ser. {\bf{1035}} (2018), 012005.
	
	

\bibitem{T03} Tod, K. P.: Isotropic cosmological singularities: other matter models. Class. Quantum Grav. {\bf 20} (2003), 521--534.


	\bibitem{Wal} Walender, P.: On the temperature jump in a rarefied gas, Ark, Fys. {\bf 7}  (1954), 507-553.
\end{thebibliography}

\end{document}